\documentclass{article}

\usepackage{arxiv}

\usepackage[utf8]{inputenc} % allow utf-8 input
\usepackage[T1]{fontenc}    % use 8-bit T1 fonts
\usepackage{hyperref}       % hyperlinks
\usepackage{url}            % simple URL typesetting
\usepackage{booktabs}       % professional-quality tables
\usepackage{amsfonts}       % blackboard math symbols
\usepackage{nicefrac}       % compact symbols for 1/2, etc.
\usepackage{microtype}      % microtypography
\usepackage{amsmath,amssymb,mathtools}
\usepackage{amsthm}
\theoremstyle{definition}
\newtheorem{rmk}{Remark}

\usepackage{cleveref}       % smart cross-referencing
\usepackage{lipsum}         % Can be removed after putting your text content
\usepackage{graphicx}
\usepackage{natbib}
\usepackage{doi}
\usepackage{xcolor}

\title{Posterior Consistency for Recovering Initial States in Nonlinear Subdiffusion Equations}

\author{ 
Haoyu Lu \\
	School of Mathematics and Statistics\\
	Xi'an Jiaotong University\\
	Xi'an, China 710049 \\
	\texttt{luhaoyu@stu.xjtu.edu.cn} \\
\And
	Shaokang Zu \\
	School of Mathematics and Statistics\\
	Xi'an Jiaotong University\\
	Xi'an, China 710049 \\
	\texttt{incredit1@stu.xjtu.edu.cn} \\
\And
    Junxiong Jia\thanks{Corresponding author} \\
	School of Mathematics and Statistics\\
	Xi'an Jiaotong University\\
	Xi'an, China 710049 \\
	\texttt{jjx323@xjtu.edu.cn} 
}

\renewcommand{\headeright}{}
\renewcommand{\shorttitle}{}

\hypersetup{
pdftitle={A template for the arxiv style},
pdfsubject={q-bio.NC, q-bio.QM},
pdfauthor={David S.~Hippocampus, Elias D.~Striatum},
pdfkeywords={First keyword, Second keyword, More},
}
\usepackage{thmtools}
\newtheorem{theorem}{Theorem}
\newtheorem{lemma}{Lemma}
\newtheorem{assumption}{Assumption}
\newcommand{\EE}{\mathbb{E}}
\newcommand{\PP}{\mathbb{P}}

\begin{document}
\maketitle

\begin{abstract}
We study the Bayesian recovery of the initial state in a semilinear time-fractional subdiffusion equation from noisy random space-time point observations. A rescaled Gaussian prior based on a Whittle--Mat\'ern process is assigned to the unknown initial condition. We prove that the solution has \(\dot H^3\)-regularity when the nonlinearity satisfies a Lipschitz condition in the Sobolev norm. We then establish posterior contraction rates for the prediction error in the \(L^2\)-norm and for the parameter in Sobolev norms. The rates are polynomial in the sample size, with exponent depending on the prior smoothness and the spatial dimension.  Moreover, we prove a minimax lower bound by constructing a wavelet-packing set and controlling the Kullback--Leibler divergences. 
\end{abstract}

% keywords can be removed

\section{Introduction}\label{sec:introduction}

Inverse problems arise when we seek to recover an unknown quantity from indirect and noisy measurements. In many applications, including imaging, medicine, materials science, and engineering, the relationship between the unknown parameter and the observed data is governed by partial differential equations (PDEs) \cite{Stuart2010, kaipio2005statistical}. A typical statistical formulation consists of observing the solution of the PDE at finitely many space-time points, with the observations corrupted by random noise \cite{kaipio2005statistical}. 

Bayesian methods provide a natural framework for such statistical inverse problems. A prior distribution is assigned to the unknown parameter, and the data update this prior to a posterior distribution via Bayes' formula. The posterior distribution provides point estimators, such as posterior means, and also gives a way to quantify uncertainty in the reconstruction. The Bayesian approach to inverse problems in infinite-dimensional spaces has been systematically developed in the last two decades; see, for example, \cite{Stuart2010,DashtiStuart2017}. From the frequentist perspective, a central question is whether the posterior distribution is consistent: when the data are generated by a fixed true parameter $\theta_0$, does the posterior concentrate around $\theta_0$ as the sample size tends to infinity? A refined version of this question asks for the rate of posterior contraction; see \cite{GhosalVaart2017,gine2021mathematical} for general background on posterior contraction theory.

For nonlinear PDE inverse problems, posterior consistency is particularly delicate. General Bayesian recovery results for PDE models were developed in \cite{Nickl2020Convergence}; related computational and variational aspects for PDE-constrained statistical models have also been studied in recent work.

In recent years, Bayesian approaches have been increasingly applied to a wide range of inverse problems for PDEs. For parabolic equations, Giordano proved a Bernstein--von Mises theorem for recovering the initial heat state in the heat equation using Gaussian series priors\cite{giordano2025bayesian}, and Kekkonen established consistency and optimal contraction rates for recovering the absorption coefficient in the heat equation \cite{Kekkonen2022}. Recently, Kow and Wang established consistency for the inverse problem of determining an unknown potential in a subdiffusion equation \cite{KowWang2025}. Motivated by these developments, the present paper concerns a related but different class of evolution equations, namely time-fractional nonlinear subdiffusion equations. 

The main contributions of this paper are as follows. First, we improve the regularity estimates for the nonlinear subdiffusion forward problem. Second, we establish posterior contraction in the prediction norm and then transfer it to Sobolev norms of the initial state using a conditional stability estimate and interpolation. Third, we prove a minimax lower bound by constructing a wavelet packing of the parameter space and bounding the Kullback--Leibler divergence between the corresponding statistical experiments. These results complement the existing theory for parabolic coefficient recovery \cite{Kekkonen2022} and potential recovery in subdiffusion equations \cite{KowWang2025}.

The rest of this paper is organized as follows. Section~\ref{sec:setup} presents the subdiffusion model and the observation scheme. Section~\ref{sec:forward} collects the forward estimates and stability inequalities to be used in the subsequent analysis. Section~\ref{sec:posterior} introduces the Gaussian prior measure, proves the posterior contraction theorem, and establishes the corresponding minimax lower bound.

\section{Model, notation, and Bayesian formulation}\label{sec:setup}

Throughout the paper, $C$ denotes a positive constant whose value may change from line to line. 
We write $a_N\lesssim b_N$ if $a_N\le Cb_N$ for all sufficiently large $N$, and $a_N\simeq b_N$ if both $a_N\lesssim b_N$ and $b_N\lesssim a_N$ hold. 
For random variables $Z_N$ and deterministic positive numbers $a_N$, the notation $Z_N=O_{\PP}(a_N)$ means that, for every $\epsilon>0$, there exists $M_\epsilon<\infty$ such that $\PP(|Z_N|>M_\epsilon a_N)<\epsilon$ for all sufficiently large $N$. 
Let $\Omega=(0,1)^d$ with $1\le d\le 3$ and $A=-\Delta$ be the negative  Laplacian with homogeneous Dirichlet boundary conditions on $\Omega$. 
Denote its eigenpairs by $\{(\lambda_j,e_j)\}_{j\ge1}$, where $\{e_j\}_{j\ge1}$ forms an orthonormal basis of $L^2(\Omega)$. We use the Hilbert scale
\(
\dot H^s(\Omega)=\left\{v\in L^2(\Omega):\|v\|_{\dot H^s}^2:=\sum_{j=1}^\infty \lambda_j^s |(v,e_j)_{L^2}|^2<\infty\right\}
\) for any \(s\ge 0\).
\color{black}
These spaces satisfy
\(
\|A^k v\|_{\dot H^s}
=\|v\|_{\dot H^{s+2k}}
\)
 for \(s,k\ge0\) and \(v\in\dot H^{s+2k}\). \color{black}
 For integer or fractional Sobolev spaces we also use the standard notation $H^s(\Omega)$ and its dual space \(H^{-s}\). \color{black}
 We also have
\(
\dot H^0=L^2,
\dot H^1=H_0^1,
\dot H^2=H^2\cap H_0^1,
\)
with
\(
\|v\|_{\dot H^1}=\|\nabla v\|_{L^2}\asymp\|v\|_{H^1}, v\in H_0^1,
\)
and
\(
\|v\|_{\dot H^2}=\|Av\|_{L^2}\asymp\|v\|_{H^2}, v\in H^2\cap H_0^1.
\) For \(s\ge0\), let \(\dot H^{-s}(\Omega)\) denote the dual of
\(\dot H^s(\Omega)\) with respect to the \(L^2(\Omega)\) pairing, equipped with the norm
\(\|h\|_{\dot H^{-s}}:=\sup_{\|v\|_{\dot H^s}=1}|\langle h,v\rangle|.\)
\color{black}
 Throughout, let $\alpha\in(0,1)$ be the fractional order and $T>0$ be fixed. For an initial state $\theta$, we consider the semilinear time-fractional diffusion equation
\begin{equation}\label{eq:PDE}
\left\{
\begin{aligned}
\partial_t^\alpha u+Au&=f(u), && \text{ in } (0,T]\times\Omega,\\
u&=0, && \text{ on }(0,T]\times\partial\Omega,\\
u(0)&=\theta, && \text{ in } \Omega.
\end{aligned}
\right.
\end{equation}
where \(f(u)\) and \(u(0)=\theta\) represent the nonlinear source term and the initial value, respectively. 
Here $\partial_t^\alpha$ denotes the Djrbashian--Caputo derivative
\(
\partial_t^\alpha u(t)=\frac{1}{\Gamma(1-\alpha)}\int_0^t (t-s)^{-\alpha}u'(s)\,ds,
\)
where $\Gamma$ is Euler's Gamma function. 
For each admissible $\theta$, let $u_\theta$ denote the corresponding solution and define the forward map as 
\(\mathcal{G}(\theta)=u_\theta.\)
By means of the Laplace transform, the mild solution of the semilinear
problem can be represented as
\begin{equation}\label{eq:mild-solution}
    u_\theta(t)
    = F(t)\theta + \int_0^t E(t-s) f(u(s))\,ds.
\end{equation}
Here \(F(t)\) and \(E(t)\) are the linear solution operators defined by
\(
    F(t)
    = \frac{1}{2\pi \mathrm{i}}
    \int_{\Gamma_{\beta,\tau}}
    e^{zt} z^{\alpha-1} (z^\alpha + A)^{-1}\,dz, 
    E(t)
    = \frac{1}{2\pi \mathrm{i}}
    \int_{\Gamma_{\beta,\tau}}
    e^{zt} (z^\alpha + A)^{-1}\,dz .
\)
The contour \(\Gamma_{\beta,\tau}\subset\mathbb{C}\) is given by
\(
    \Gamma_{\beta,\tau}
    =
    \{z\in\mathbb{C}: |z|=\tau,\ |\arg z|\le \beta\}
    \cup
    \{z\in\mathbb{C}: z=\rho e^{\pm \mathrm{i}\beta},\ \rho\ge \tau\},
\)
where \(\tau\ge 0\) and \(\pi/2<\beta<\pi/\alpha\), oriented counterclockwise.
According to \cite[Theorems 6.4 and 3.2]{jin2021fractional}, the operators
\(F(t)\) and \(E(t)\) satisfy the following smoothing estimate: for all
\(t>0\),
\begin{equation}\label{eq:smoothing-FE}
    \|A^\nu F(t)v\|_{\dot H^p}
    + t^{1-\alpha}\|A^\nu E(t)v\|_{\dot H^p}
    \le C \min\{t^{-\alpha},t^{-\nu\alpha}\}
    \|v\|_{\dot H^p},
    \quad 0\le \nu\le 1,\quad p\in\mathbb{R},
\end{equation}
which will be used in Section \ref{sec:forward}.

Let $(t_i,x_i)$, $i=1,\ldots,N$, be independent random design points uniformly distributed on $(0,T]\times\Omega$. We observe
\(    Y_i=\mathcal{G}(\theta)(t_i,x_i)+\epsilon_i,
    \ i=1,\ldots,N,
\) where $\epsilon_i\overset{\mathrm{iid}}{\sim}N(0,1)$ and the noises are independent of the design points. We write
\(
    D_N=\{(Y_i,t_i,x_i):1\le i\le N\}
\)
for the full data set. Let \(\PP_\theta^{(i)}\) denote the joint law of the \(i\)-th observation
\((Y_i,t_i,x_i)\) under parameter \(\theta\). The corresponding joint law of \(D_N\) is denoted by $\PP^N_{\theta
}$.
The log-likelihood, up to an additive constant independent of $\theta$, is
\(
    \ell_N(\theta)=-\frac12\sum_{i=1}^N\left(Y_i-\mathcal{G}(\theta)(t_i,x_i)\right)^2.
\)
For a prior distribution $\Pi$ on the parameter space, the posterior law of \(\theta|D_N\) is given by the Bayes' Formula
\begin{align}\label{BayesFormula}
    \frac{d\Pi(\theta|D_N)}{d\Pi}=\frac{e^{\ell_N(\theta)}}{\int e^{\ell_N(\vartheta)}d\Pi(\vartheta)},
\end{align}
which we will analyze to derive upper and lower bounds for the posterior contraction rate.

\section{Forward estimates and stability}\label{sec:forward}

The posterior analysis relies on several analytic properties of the forward map. We first provide a higher-order regularity estimate, which will be used in the interpolation step of the stability argument. 
\begin{theorem}[Higher-order regularity]\label{Higher-order regularity} Let {\color{black}
\(\theta\in B_2(R)\), where \(B_s(R):=\{v:v\in \dot{H}^s:\|v\|_{\dot{H}^s}\le R\}\) for a given radius \(R>0\).} Assume the
nonlinearity satisfies \(f\in C^1,f(0)=0\) and the following Lipschitz condition
\begin{align}\label{fLip20}
    \|f(v_1)-f(v_2)\|_{\color{black}\dot{H}^\kappa}
\le L(M)\|v_1-v_2\|_{\color{black}\dot{H}^\kappa},\quad v_1,v_2\in B_\kappa(M),\kappa\in [0,1].
\end{align}
Then the solution satisfies the
higher-order regularity estimate
\begin{equation}\label{neq:HighOrderRegularity}
    \|u_\theta(t)\|_{\color{black}\dot H^{2+\kappa}}
    \le C(\Omega,T,\alpha,\kappa,R,{\color{black}L_{\Omega,\kappa,T,R}})
    t^{-\kappa\alpha/2}\|\theta\|_{\color{black}\dot H^2},
    \qquad 0<t\le T.
\end{equation}
\end{theorem}

\begin{proof}
\color{black}
For \(\theta\in B_2(R)\), the assumptions on \(f\) imply that
\(f(\theta)\in H_0^1(\Omega)\). 
By \cite[Lemma 2.3]{wu2025numerical},
\(\sup_{0<t\le T}\|u_\theta(t)\|_{\dot H^2}
\le C_T\|\theta\|_{\dot H^2}\).
Combining this estimate with the Sobolev embedding,
we obtain \(u_\theta(t)\in 
B_\kappa(C_TC_{\Omega,\kappa}R)\)
for all \(0<t\le T\).
Hence 
\eqref{fLip20} applies to the solutions with constant \(L=L_{T,\Omega,\kappa,R}\).
Taking \(v_2=0\) and using \(f(0)=0\), we obtain
\begin{equation}\label{fLip2}
    \|f(u_\theta)\|_{\dot H^\kappa}
    \le L\|u_\theta\|_{\dot H^\kappa}
    \le C_{\Omega,\kappa}L\|u_\theta\|_{\dot H^2}.
\end{equation} 
This estimate and the constant \(L\) will be used repeatedly below. 
\color{black}
 For \(\theta\in\color{black}\dot H^2(\Omega)\),
the solution admits the representation \eqref{eq:mild-solution}. Splitting the integral over \([0,t/2]\)
and \([t/2,t]\), and applying \eqref{eq:smoothing-FE}, we obtain,
for \(0\le\kappa\le1\),
\begin{align*}
\|u_\theta(t)\|_{\color{black}\dot H^{2+\kappa}}
&\le
\|A^{\kappa/2}F(t)\theta\|_{\color{black}\dot H^2}
+
\int_0^{t/2}
\|AE(t-s)f(u(s))\|_{\color{black}\dot H^\kappa}\,ds
+
\left\|
\int_{t/2}^t AE(t-s)f(u(s))\,ds
\right\|_{\color{black}\dot H^\kappa} \\
&\le
C_1t^{-\kappa\alpha/2}\|\theta\|_{\color{black}\dot H^2}
+
C_1\int_0^{t/2}(t-s)^{-1}
\|f(u(s))\|_{\color{black}\dot H^\kappa}\,ds
+
R(t) \\
&\le
C_1t^{-\kappa\alpha/2}\|\theta\|_{\color{black}\dot H^2}
+
C_1\ln 2\sup_{0<s\le T}\|f(u(s))\|_{\color{black}\dot H^\kappa}
+
R(t),
\end{align*}
where
\(
R(t):=
\left\|
\int_{t/2}^t AE(t-s)f(u(s))\,ds
\right\|_{\color{black}\dot H^\kappa}.
\)
\color{black}By \eqref{fLip2}, we have
\(\sup_{0<s\le T}\|f(u(s))\|_{\dot H^\kappa}
\le C_{\Omega,\kappa,T,L,R}\|\theta\|_{\dot H^2}\).\color{black}
 Using {\color{black}\(F'(t)=-AE(t)\)} and \(F(0)=I\) from
\cite{jin2021fractional}, we have
\(\partial_sF(t-s)=AE(t-s)\).
Integration by parts yields
\[
R(t)
\le
\|f(u(t))\|_{\color{black}\dot{H}^\kappa}
+
\|F(t/2)f(u(t/2))\|_{\color{black}\dot{H}^\kappa} 
+
\int_{t/2}^t
\|A^{\kappa/2}F(t-s)f'(u(s))u'(s)\|_{L^2}\,ds .
\]
Using the smoothing estimates \eqref{eq:smoothing-FE}, \(
f\in C^1
\),
and 
\(
\|u'(s)\|_{L^2}
\le C_2 s^{\alpha-1}\|\theta\|_{\color{black}\dot{H}^2}
\) from \cite[Lemma 2.3]{wu2025numerical}, 
we obtain
\[
R(t)
\lesssim
2\sup_{0<s<T}\|f(u(s))\|_{\color{black}\dot{H}^\kappa}
+
C_2L\|\theta\|_{\color{black}\dot{H}^2}
\int_{t/2}^t (t-s)^{-\alpha}s^{\alpha-1}\,ds. 
\]
Using again the bound
\(
\sup_{0<s\le T}\|f(u(s))\|_{\color{black}\dot H^\kappa}
\le C_{\Omega,\kappa,T,L,R}\|\theta\|_{\color{black}\dot H^2}
\)
and evaluating the integral, we obtain \eqref{neq:HighOrderRegularity}
with a constant depending only on \(\Omega,T,L,\alpha,\kappa,R\). 
\end{proof}
The higher-order estimate above will be used in the interpolation
argument for conditional stability. We next establish the boundedness,
forward Lipschitz continuity, and conditional stability 
needed for the posterior analysis.
\color{black}
 Let \(\mathcal R=H^{\bar\gamma}(\Omega)\cap H_0^1(\Omega)\), where
\(\bar\gamma\ge2\), equipped with the norm
\(\|\cdot\|_{\mathcal R}:=\|\cdot\|_{H^{\bar\gamma}(\Omega)}\).
\color{black}

\begin{lemma}[Forward estimates]\phantomsection\label{lemma2}
Let \(f\) satisfy all the conditions required in Theorem \ref{Higher-order regularity}. 
Let {\color{black}\(\theta,\vartheta\in \mathcal R\)}. Then the following estimates hold.

\begin{enumerate}
\item[(i)] For every \(M>0\), the solutions satisfy
\[
\sup_{\|\theta\|_{\mathcal R}\le M}
\sup_{0<t\le T}\|u_\theta(t)\|_{L^\infty}<\infty.
\]

\item[(ii)] For every \(\mu\in[0,1]\), the solutions satisfy
\begin{equation}\label{lemma2Lip}
    \|u_\theta-u_\vartheta\|_{L^2(0,T;L^2)}
    \le C(\Omega,T,\alpha,\mu,{\color{black}L_{T,\Omega,\kappa,R}})\|\theta-\vartheta\|_{\color{black}\dot H^{-\mu}}.
\end{equation}

\item[(iii)] There exists
\(T_*\in(0,T]\) such that for every \(\mu\in(0,1], \kappa\in[0,1]\) and \(M>0\),
\begin{equation}\label{eq:stability}
    \|\theta-\vartheta\|_{\color{black}\dot H^{-\mu}}
    \le C(M,\Omega,T_*,\kappa,\alpha,\mu,{\color{black}L_{T,\Omega,\kappa,R}})
    \|u_\theta-u_\vartheta\|_{L^2(0,T_*;L^2)}
    ^{\frac{\kappa+\mu}{2+\kappa}},
    \qquad
    \|\theta\|_{\mathcal R},\|\vartheta\|_{\mathcal R}\le M.
\end{equation}
\end{enumerate}
\end{lemma}

\begin{proof}
\color{black}
For (i), \cite[Lemma 2.3]{wu2025numerical} with \(p=2\) gives
\(
\sup_{0<t\le T}\|u_\theta(t)\|_{\dot H^2}
\le C_T\|\theta\|_{\dot H^2}.
\)
Since \(d\le3\) and \(\bar\gamma\ge2\), the embeddings
\(\dot H^2\hookrightarrow H^2\hookrightarrow L^\infty\) and
\(\mathcal R\hookrightarrow\dot H^2\) imply
\[
\sup_{\|\theta\|_{\mathcal R}\le M}\sup_{0<t\le T}
\|u_\theta(t)\|_{L^\infty}
\le C\sup_{\|\theta\|_{\mathcal R}\le M}\|\theta\|_{\mathcal R}
\le CM<\infty.
\]
Here \(C\) is independent of \(\theta\) and \(M\).
\color{black}
 For (ii), \cite[Lemma 2.4]{wu2025numerical} gives
\(
\|u_\theta(t)-u_\vartheta(t)\|_{L^2}
\le C_{T,L,\alpha,\mu}t^{-\alpha\mu/2}
\|\theta-\vartheta\|_{\color{black}\dot H^{-\mu}}\), for \(0<t\le T\). 
Squaring and integrating over \((0,T)\), we obtain
\[
\|u_\theta-u_\vartheta\|_{L^2(0,T;L^2)}^2
\le C\|\theta-\vartheta\|^2_{\color{black}\dot H^{-\mu}}
\int_0^T t^{-\alpha\mu}\,dt
=\frac{CT^{1-\alpha\mu}}{1-\alpha\mu}
\|\theta-\vartheta\|^2_{\color{black}\dot H^{-\mu}}
.
\]
Since \(\alpha\mu<1\),
taking square roots yields \eqref{lemma2Lip}.

For (iii), fix \(\mu\in(0,1]\).
By \cite[Theorem 2.1]{wu2025numerical}, choosing
\(T_*\in(0,T]\) sufficiently small, we have
\(
\|\theta-\vartheta\|_{\color{black}\dot H^{-\mu}}
\le C\|u_\theta(t)-u_\vartheta(t)\|_{\color{black}\dot H^{2-\mu}}
\)
for \(0<t\le T_*\), with a constant independent of
\(t,\theta,\vartheta\). 
By Theorem~\ref{Higher-order regularity} and the embedding
\(\mathcal R\hookrightarrow\dot H^2\),
\(
\|u_\theta(t)-u_\vartheta(t)\|_{\color{black}\dot H^{2+\kappa}}
\le Ct^{-\alpha\kappa/2}
(\|\theta\|_{\color{black}\dot H^2}+\|\vartheta\|_{\color{black}\dot H^2})
\le CMt^{-\alpha\kappa/2}.
\)
Interpolation between \(L^2(\Omega)\) and \(\color{black}\dot H^{2+\kappa}(\Omega)\) yields
\[\|\theta-\vartheta\|_{\color{black}\dot H^{-\mu}}
\le
C\|u_\theta(t)-u_\vartheta(t)\|_{L^2}
^{\frac{\kappa+\mu}{2+\kappa}}
\|u_\theta(t)-u_\vartheta(t)\|
^{\frac{2-\mu}{2+\kappa}}_{\color{black}\dot{H}^{2+\kappa}}
\le
CM^{\frac{2-\mu}{2+\kappa}}\|u_\theta(t)-u_\vartheta(t)\|_{L^2}
^{\frac{\kappa+\mu}{2+\kappa}}
t^{-\frac{\alpha\kappa}{2}\frac{2-\mu}{2+\kappa}}.
\]
Integrating over \((0,T_*)\) and applying H\"older's inequality gives
\begin{align*}
\|\theta-\vartheta\|_{\color{black}\dot{H}^{-\mu}}
\le&
\frac{C}{T_*}M^{\frac{2-\mu}{2+\kappa}}
\int_0^{T_*}
\|u_\theta(t)-u_\vartheta(t)\|_{L^2}
^{\frac{\kappa+\mu}{2+\kappa}}
t^{-\frac{\alpha\kappa}{2}\frac{2-\mu}{2+\kappa}}dt\\
\leq &\frac{C}{T_*}M^{\frac{2-\mu}{2+\kappa}}
\|u_\theta-u_\vartheta\|_{L^2(0,T_*;L^2(\Omega))}
^{\frac{\kappa+\mu}{2+\kappa}}
\left[\int_0^{T_*}
t^{-\frac{\alpha\kappa(2-\mu)}{2(2+\kappa)-\kappa-\mu}}dt\right]^{\frac{2(2+\kappa)-\kappa-\mu}{2(2+\kappa)}}.
\end{align*}
Since
\(
\frac{\alpha\kappa(2-\mu)}{2(2+\kappa)-\kappa-\mu}<\frac{2}{3}<1,
\)
the integral is finite. Absorbing the remaining factors into
the constant proves \eqref{eq:stability}. 
\end{proof}
\begin{rmk}
This lemma provides the analytic estimates needed for the posterior
contraction theorem below. Parts (i) and (ii) are used to establish
contraction in the prediction norm, while part (iii) transfers
prediction contraction to parameter contraction. 
In what follows, we use \(\mu=\kappa=1\), giving a stability
exponent of \(2/3\) in \eqref{eq:stability}.
{\color{black}Since
\(\dot H^{-1}(\Omega)=H^{-1}(\Omega)=(H_0^1(\Omega))^*\)
with equivalent norms, parts (ii) and (iii) also hold
with the \(H^{-1}(\Omega)\)-norm.}
\end{rmk}

\section{Upper and lower bounds of the posterior contraction rate }\label{sec:posterior}
In this section we denote by $\mathcal{L}(Z)$ the law of a random variable $Z$.
To arrive at the general posterior contraction theorem, we require the following assumptions on the prior:
\begin{assumption}\label{Assumption1}
Let \(\Pi'\) be a centred Gaussian Borel probability measure on the linear space
\(\Theta\subset L^2(\Omega)\) with RKHS \(\mathcal H\).
Suppose further that \(\Pi'(\mathcal R)=1\) for some separable normed linear
subspace \((\mathcal R,\|\cdot\|_{\mathcal R})\) of \(\Theta\).
\end{assumption}

For \(\gamma>d/2\), we consider a Whittle--Matérn process
\(\mathcal M=\{\mathcal M(x):x\in\Omega\}\).
By \citep[Theorem B.1.3]{nickl2023bayesian}, the RKHS of
\(\mathcal L(\mathcal M)\) is \(H^\gamma(\Omega)\), and
\(\mathcal L(\mathcal M)\) is supported on \(H^{\bar\gamma}(\Omega)\)
for \(0\le\bar\gamma<\gamma-d/2\).
We further assume that \(\bar\gamma\ge2\) and \(\bar\gamma>1+d/2\).
By the Sobolev embedding theorem, \(\mathcal M\) admits a version
with \(C^1\) sample paths.

Assume that the true parameter {\color{black}\(\theta_0\in H^{\gamma}(\Omega)\)}
has compact support in \(\Omega\).
Choose a compact set \(\mathcal K\Subset\Omega\) containing
\(\operatorname{supp}(\theta_0)\).
To enforce boundary conditions on the prior, choose a smooth cut-off function
\(\chi\in C_0^\infty(\Omega)\) such that \(\chi=1\) on \(\mathcal K\).
Then \(\Pi':=\mathcal L(\chi\mathcal M)\) is a centred Gaussian Borel probability
measure supported on \(H^{\bar\gamma}(\Omega)\cap H_0^1(\Omega)\).
Its RKHS is
\(\mathcal H=\{\chi F:F\in H^\gamma(\Omega)\}\).
This satisfies Assumption~\ref{Assumption1}, with
\(\mathcal R=H^{\bar\gamma}(\Omega)\cap H_0^1(\Omega)\),
equipped with the \(H^{\bar\gamma}(\Omega)\)-norm, and
\(\Theta=L^2(\Omega)\).
 Finally, let \(\Pi_N=\mathcal{L}\left(\frac{1}{\sqrt N\,\delta_N}\chi\mathcal M\right)\), where \(\delta_N>0\)
will be chosen later.
Its RKHS \(\mathcal H_N\) coincides with \(\mathcal H\) as a set,
with the rescaled norm
\(\|\cdot\|_{\mathcal H_N}
=\sqrt N\,\delta_N\|\cdot\|_{\mathcal H}\). We now state the upper bound.
\begin{theorem}[Posterior contraction]
\label{thm:posterior_contraction}
Suppose that Assumption \ref{Assumption1} and assumptions of Theorem \ref{Higher-order regularity} and Lemma \ref{lemma2} hold. Then for all \(b>0\) there exists $L>0$ such that
\begin{equation}\label{eq:prediction_contraction}
\PP^N_{\theta_0}\left(
\Pi_N\left(
\theta:\|\theta\|_{H^{\bar{\gamma}}}\le L,
\ \|u_\theta-u_{\theta_0}\|_{L^2(0,T;L^2)}\le L\delta_N
\ \middle|\ D_N
\right)\leq 1-e^{-bN\delta_N^2}\right)
\to 0, N\to \infty,
\end{equation}
where
\(    \delta_N=N^{-\frac{\gamma+1}{2\gamma+2+d}}.
\)
Consequently, for all  \(b>0\) there exists $L>0$ such that
\begin{equation}\label{eq:parameter_contraction}
\PP^N_{\theta_0}\left(
\Pi_N\left(
\theta:\|\theta-\theta_0\|_{H^\xi(\Omega)}\le L\widetilde\delta_N(\xi)
\ \middle|\ D_N
\right)\leq 1-e^{-bN\delta_N^2}\right)
\to 0, N\to \infty,
\end{equation}
where \(\widetilde\delta_N(\xi)=\delta_N^{\frac23\frac{\bar{\gamma}-\xi}{\bar{\gamma}+1}}\) and \(
-1\le \xi\le \bar{\gamma}\).
Moreover, whenever the posterior mean is well defined, we have
\begin{equation}\label{eq:posterior_mean_rate}
\PP^N_{\theta_0}\left(
\left\|\EE^{\Pi_N}(\theta\mid D_N)-\theta_0\right\|_{H^\xi(\Omega)}>L\widetilde\delta_N(\xi)\right)\to 0, N\to\infty.
\end{equation}
\end{theorem}

\begin{proof}
We apply \citep[Theorem~2.2.2]{nickl2023bayesian}
with the parameter and regularisation spaces defined above. In view of parts (i) and (ii) of Lemma~\ref{lemma2}, we can verify \citep[Condition 2.1.1]{nickl2023bayesian} for \(\kappa = \mu=1, \mathcal{X} = (0,T]\times\Omega,\) and \(V = W=\mathbb R\). Furthermore, \citep[Condition 2.2.1]{nickl2023bayesian} with such \(\mathcal{R}\) follows from Assumption \ref{Assumption1}.
 Thus {\color{black}\eqref{eq:prediction_contraction}} follows, and it remains to transfer the prediction rate to the parameter space. The   interpolation inequality gives 
\[
\|\theta-\theta_0\|_{H^\xi}
\lesssim
\|\theta-\theta_0\|_{H^{-1}}^m
\|\theta-\theta_0\|_{H^{\bar{\gamma}}}^{1-m},
\qquad
m=\frac{\bar{\gamma}-\xi}{\bar{\gamma}+1},\quad -1\le\xi\le\bar{\gamma} .
\]
For parameters in the set appearing in
\eqref{eq:prediction_contraction}, both $\theta$ and $\theta_0$ are bounded in $H^{\bar{\gamma}}$, hence the stability estimate \eqref{eq:stability} with \(\mu=\kappa=1\) gives
\(
\|\theta-\theta_0\|_{H^\xi}
\lesssim
\|u_\theta-u_{\theta_0}\|_{L^2(0,T;L^2)}^{m{\frac{\kappa+\mu}{2+\kappa}}}
\lesssim \delta_N^{{\frac{2}{3}}\frac{\bar{\gamma}-\xi}{\bar{\gamma}+1}},
\)
which completes the proof of \eqref{eq:parameter_contraction}. 
The posterior mean bound \eqref{eq:posterior_mean_rate} follows from the same contraction estimate together with standard posterior moment arguments and Fernique-type integrability for the Gaussian prior, see the proof of \citep[Theorem 6]{giordano2020consistency} for details.
\end{proof}

We finally show that the recovery of the initial state cannot be faster than a minimax lower bound. Consider the parameter set
\(
\Theta_{\gamma}(B)=\{\theta\in H^{\gamma}(\Omega):\|\theta\|_{H^{\gamma}}\le B\}\) for a fixed \(B>0\).

\begin{theorem}[Minimax lower bound]\label{thm:lower_bound}
Assume the forward Lipschitz estimate \eqref{lemma2Lip}. Let $-\mu\le\xi<{\color{black}\gamma},\mu\in (0,1]$. 
\color{black}For every sufficiently small \(\epsilon>0\), there exists
\(C_\epsilon>0\) such that\color{black}
\begin{equation}\label{eq:minimax_theta}
\liminf_{N\to\infty}
\inf_{\widehat\theta_N}
\sup_{\theta\in\Theta_{{\color{black}\gamma}}(B)}
\PP^N_\theta\left(
\|\widehat\theta_N-\theta\|_{H^\xi(\Omega)}
>{\color{black}C_\epsilon} N^{-\frac{{\color{black}\gamma}-\xi}{2{\color{black}\gamma}+2\mu+d}}
\right)
\ge 1-\epsilon,
\end{equation}
where the infimum is taken over all measurable estimators $\widehat\theta_N=\widehat\theta_N(D_N)$.
\end{theorem}
\begin{proof}
The proof follows the standard testing reduction, for example \citep[Theorem~6.3.2]{gine2021mathematical}. We construct a finite family of parameters that are well separated in $H^\xi$ while the corresponding \(\PP^N\) remain close in Kullback--Leibler divergence.

Given \(j\), let $\{\Psi_{j,r}\}_{r=1}^{n_j}$ be compactly supported Daubechies wavelets with mutually disjoint supports contained in $\Omega$, where $n_j\simeq 2^{jd}$. 
For $b_{m,\cdot}\in\{-1,1\}^{n_j}$ define {\color{black}
\(
    \theta_m(x)=\kappa 2^{-j(\gamma+d/2)}\sum_{r=1}^{n_j} b_{m,r}\Psi_{j,r}(x).
\)}
 For sufficiently small $\kappa>0$, all $\theta_m$ belong to $\Theta_{{\color{black}\gamma}}(B)$. By the Varshamov--Gilbert bound in \citep[Example 3.1.4]{gine2021mathematical}, one can choose \(b_{m,r}\in \{-1,1\},1\leq m\leq M_j\) with $M_j\ge 3^{n_j/4}$ such that
\(
    \sum_{r=1}^{n_j}(b_{m,r}-b_{m',r})^2\gtrsim n_j,
\) for any \(m\ne m'.\)
With this inequality and the wavelet characterisation of Sobolev norms, we have
\begin{equation}\label{eq:wavelet_sep}
    \|\theta_m-\theta_{m'}\|_{H^\xi}^2
    \simeq 
    \sum_{r}2^{2j\xi}\langle \theta_m-\theta_{m'},\Psi_{j,r}\rangle_{L^2}^2
    \gtrsim 2^{2j\xi}\kappa^2 2^{-2j({\color{black}\gamma}+d/2)}n_j
    =\kappa^2 2^{-2j({\color{black}\gamma}-\xi)},
    \qquad m\ne m'.
\end{equation}

On the other hand, the Kullback--Leibler divergence satisfies
\[
\mathrm{KL}(\PP^N_{\theta_m},{\color{black}\PP^N_0})
=\sum_{i=1}^N\EE_{\theta_m}
\ln\frac{d\PP^{(i)}_{\theta_m}}{d{\color{black}\PP^{(i)}_0}}
=\frac12\sum_{i=1}^N\EE_{\theta_m}
\bigl(u_{\theta_m}(t_i,x_i)-u_0(t_i,x_i)\bigr)^2
={\color{black}\frac{N}{2T}}
\|u_{\theta_m}-u_0\|_{L^2(0,T;L^2)}^2,
\]
here the subscript \(0\) corresponds to the initial condition \(\theta=0\).
Using the Lipschitz property \eqref{lemma2Lip} and arguing as in \eqref{eq:wavelet_sep}, we obtain
\(
\|u_{\theta_m}-u_{0}\|_{L^2(0,T;L^2(\Omega))}^2
\lesssim \|\theta_m\|_{H^{-\mu}}^2
\lesssim \kappa^2 2^{-2j({\color{black}\gamma}+\mu)}.
\)
Choosing $2^j\simeq N^{1/(2{\color{black}\gamma}+2\mu+d)}$, we obtain
\begin{align}\label{eq:KLInequality}
\mathrm{KL}(\PP^N_{\theta_m},{\color{black}\PP^N_{0}})
\lesssim \kappa^2 2^{jd}\simeq \kappa^2n_j
\lesssim \epsilon\log M_j
\end{align} 
for any small \(\epsilon>0\) by taking $\kappa$ sufficiently small.
 Thus far, we have proved that \(\{\theta_m\}_{m=1}^{M_j}\) is an \(N^{-\frac{{\color{black}\gamma}-\xi}{2{\color{black}\gamma}+2\mu+d}}\)-separated set in \(\Theta_{{\color{black}\gamma}}(B)\) in the sense of \eqref{eq:wavelet_sep}, while  the KL computation \eqref{eq:KLInequality} shows that their observational laws can be made arbitrarily close. Thus we have many candidates that are statistically indistinguishable, which is exactly what the lower bound requires. Finally \eqref{eq:minimax_theta} is given by \citep[Theorem~6.3.2]{gine2021mathematical}.
\end{proof}

\begin{rmk}
 Whether the upper and lower bounds on the contraction rate match is a challenging question, as it effectively asks for the exact rate as the sample size grows. This depends on the prior, the estimator, and the conditional stability estimate. Matching bounds were obtained in \cite{Kekkonen2022} for the posterior mean under Lipschitz stability with a truncated Gaussian prior, while \cite{KowWang2025} found that under Hölder stability the same prior and estimator do not yield matching bounds. \cite{Furuya2024} observed non-matching bounds under logarithmic stability, whereas \cite{Nickl2020Convergence} proved matching bounds for the MAP estimator under Lipschitz stability with a Gaussian prior. Whether matching bounds hold for the subdiffusion problem studied here will be an interesting topic for future work.
\end{rmk}

\section*{Acknowledgments}
This work was partially supported by the National Natural Science Foundation of China (Grants No. 12322116, 12271428, 12326606).

\bibliographystyle{unsrtnat}
\bibliography{references}  %%% Uncomment this line and comment out the ``thebibliography'' section below to use the external .bib file (using bibtex) .





% \end{thebibliography}

\end{document}